\documentclass[12pt,reqno]{amsart}
\usepackage{amssymb}
\begin{document}
\newcommand{\Z}{\mathbb{Z}}
\newcommand{\aut}{\mathrm{Aut}\, R}
\newcommand{\obar}{\overline}
\newcommand{\abs}[1]{|#1|}
\newcommand{\F}{\mathbb{F}}
\newcommand{\Q}{\mathbb{Q}}

\theoremstyle{plain}
\newtheorem{thm}{Theorem}[section]
\newtheorem{lemma}[thm]{Lemma}
\newtheorem{cor}[thm]{Corollary}
\theoremstyle{definition}
\newtheorem{defn}{Definition}
\newtheorem{conj}{Conjecture}
\theoremstyle{remark}
\newtheorem{rmk}{Remark}
\newtheorem{expl}{Example}
\newtheorem{que}{Question}
\title[Automorphisms of rings]{A graphical representation of
  rings via automorphism groups}
\author{N. Mohan Kumar}
\address{ Department of Mathematics,
Washington University in St. Louis, St. Louis, Missouri, 63130, U.S.A.}
\email{kumar@wustl.edu}
\urladdr{http://www.math.wustl.edu/\~{}\!\!\!
kumar}
\thanks{The first author was partially supported by a grant
from NSA}
\author{Pramod K. Sharma}
\address{School Of Mathematics,
Vigyan Bhawan, Khandwa Road, INDORE--452 001, INDIA.}
\email{pksharma1944@yahoo.com}
\thanks{All correspondences will be handled by the second author}
\date{}

\begin{abstract} Let $R$ be a commutative ring
                  with identity. We define a graph $\Gamma_{\aut}(R)$  on $
                   R$, with vertices elements of $R$, such that
                   any two distinct vertices $x, y$ are
                   adjacent if and only if there exists $\sigma \in
                   \aut$ such that $\sigma(x)=y$. The idea is to
                   apply graph theory to study
                   orbit spaces of rings under automorphisms. In this
                   article, we define the notion of a ring of type $n$
                   for $n\geq 0$ and characterize all rings of type
                   zero. We also characterize local rings $(R,M) $ in
                   which either  the subset of units ($\neq 1 $) is
                   connected or the subset $M- \{0\}$ is connected in
                   $\Gamma_{\aut}(R)$.
\end{abstract}

\maketitle

                  \section{Introduction} Throughout this article,
                  all  rings are commutative with identity. We
                  denote by $ \Z_n $, the ring of integers modulo $n$,
                  and by $ U(R) $, the group of units of a ring
                  $R$. We will also use the notation $\F_q$ to denote
                  a field of $q$ elements, where of course, $q$ is the
                  power of a prime.

                   In the last decade, study of rings using
                   properties of graphs has attracted considerable
                   attention. In \cite{B}, I.~Beck  defined a simple graph
                   on a  commutative ring $ R $  with
                   vertices elements of $ R $ where two different
                   vertices $x$ and $y$ in $ R $ are adjacent, by
                   which we mean as usual that they are connected by
                   an edge,  if and
                   only if $ xy=0 $. In \cite{SB}, the authors defined
                   another graph on a ring $ R $ with vertices
                   elements of $ R $ such that two different
                   vertices $x$ and $y$ are adjacent if and only if $ Rx
                   + Ry= R $. In this article, we define another
                   graph $ \Gamma_{\aut}(R) $ with vertices elements of $ R
                   $  where two
                   different vertices $ x, y \in \Gamma_{\aut}(R) $ are
                   adjacent if and only if $ \sigma (x)= y $ for
                   some $\sigma \in \aut $. It is proved that if
                   $\Gamma_{\aut}(R) $ is totally disconnected, which is
                   equivalent to $ \deg x$  being zero for all $x \in
                   R $, then $ R $ is
                   either $ \Z_n $ or  $ \Z_2[X]/(X^2) $.
                  As usual, the degree of a vertex is the number of edges
                   emanating from it. Further,
                   we define
                   the notion of rings of type $n$ and study the
                   structure of rings of type at most one. We also
                   characterize finite local rings $(R, M)$ with
                   either  $ U(R)- \{1\}$  connected or $ M-\{0\}$  connected
                    as subsets of $ \Gamma_{\aut}(R) $.

                    In general, if for a ring $ R$,  $H $ is a
                    subgroup of $ \aut $, then we can define a
                    graph structure on $ R $ using $ H $ instead of
                    $ \aut $. We shall denote this graph by
                    $\Gamma_H(R) $. We expect that this approach may be
                    useful in the study of orbit space of $ R $
                    under $ \aut$.

                    \section{Preliminaries}
We recall some basic notions from graph theory.

                    A simple graph $\mathfrak{G}$ is a
                    non-empty set $V$ together with a set $E$ of
                    un-ordered pairs of distinct elements of $V$.
                                        The elements of $V$ are called
                    vertices and an
                    element $e=\{u,v\}\in E$ where $u,v\in V$ is
                    called an edge of $\mathfrak{G}$
                    joining the vertices $u\,\text{and}\, v$. If $
                    \{u,v\}\in
                    E$, then $u\,\text{ and }\, v$ are called adjacent
                    vertices. In this case $u$ is
                    adjacent to $v$ and $v$ is  adjacent to $u$.
                    We shall normally denote the graph just
                    by $\mathfrak{G}$  and call $|V|$, the cardinality of
                    $V$, the order of $\mathfrak{G}$. We shall
                    sometimes write $|\mathfrak{G}|$ for
                    the order of $\mathfrak{G}$. For any vertex
                    $v\in\mathfrak{G}$,
                    degree of $v$, denoted by $\deg v $, is the
                    number of edges of $\mathfrak{G}$ incident with $v$. An automorphism $\alpha$
                    of a graph $\mathfrak{G}$  is a permutation of the set of vertices $V $ of
                    $\mathfrak{G}$ which preserves adjacency.

                    A subgraph of $\mathfrak{G}$ is a graph having all its
                    vertices and edges in $\mathfrak{G}$.  A graph
                    $\mathfrak{G}$ is called complete
                    if any two vertices in $\mathfrak{G}$ are adjacent. A
                    clique of a graph is a maximal complete
                    subgraph.

                    A graph $\mathfrak{G}$ is called connected if for all
                    distinct vertices $x,y\in \mathfrak{G}$ there is a path
                    from $x$  to $ y $. A graph
                    $\mathfrak{G}$ is called
                    totally disconnected if there are no edges in
                    $\mathfrak{G}$. That is, the edge set of
                    $\mathfrak{G}$ is empty. We say that a graph  $\mathfrak{G}$ can be embedded
                    in a surface S if it can be drawn on S so that no two edges inersect. A graph
                    is planar if it can be embedded in a plane.

                    For a ring $R$, $\aut$ operates in a natural
                    way on $R$. If $S\subset \Gamma_{\aut}(R) $ is
                    connected, then for any $a,b\in S$, there is
                    $\sigma \in \aut$ such that $\sigma(a)=b$. For
                    any $x\in R$, we denote by $O(x)$ the orbit of
                    $x$ under the action of $\aut$. In fact $O(x)$ is
                    the clique of $\Gamma_{\aut}(R)$ containing $x$.
                    Moreover, any clique of $\Gamma_{\aut}(R)$ is of the
                    form $O(x)$ for some $x\in R$.

                    Let $K/k$ be a field extension. Then for any
                    subgroup $H$ of $\mathrm{Aut}\,(K)$, $k\subset
                    \Gamma_H(K)$
                    is totally disconnected if and only if $H\subset
                    \mathrm{Aut}_k(K)$.

We record some elementary
                    results.
            \begin{lemma}\label{l.1}
               Let $R$ be an integral domain
                    and $ G=\aut$. For any $\lambda \in R-R^G,$
                    $\lambda$ is integral over $ R^G $ if and only if
                    the clique of $\Gamma_{\aut}(R)$ containing $\lambda$
                    is finite.
            \end{lemma}

                     The proof is standard.
             \begin{thm}\label{t.2}
               Let $R$ be a Noetherian integral domain such
                    that $\Gamma_{\aut}(R)$ has a finite number of
                    cliques. Then $R$ is a finite field.
             \end{thm}

                     \begin{proof}
                The proof follows from \cite[Corollary 16]{S2}.
             \end{proof}

                    Next we define the notion of \textit{type} of a
                    ring $R$.

            \begin{defn}\label{d.3}
              A ring $R$ is called of type $n$ if for
                         all $x\in \Gamma_{\aut}(R)$, $\deg x \leq n$, and
                         there exists at least one $y\in
                         \Gamma_{\aut}(R)$ such that $\deg y=n$.
            \end{defn}

            \begin{rmk}\label{r.4}
               Assume that the  ring $R$ is a direct product
                          of rings
                          $A$ and $ B$. If $R$ is of type
                          $n$, then $A$ and $ B$ are of type $\leq n$.
            \end{rmk}

            \begin{expl}\label{e.5}
              For any prime $p$, $ R= \Z_{p}[X]/(X^2)$
                          is a ring of type $p-2$.

This can be seen as follows.
Let us denote by $x$ the image of $X$ in $R$. If $\psi$ is an
automorphism of $R$, then $\psi(x)=ax$ for some $0\neq a\in\Z_p$ and
conversely, given such an $a\in\Z_p$, we can define an automorphism of
$R$ by sending $x\mapsto ax$. Then, it is clear that $\aut$ has order
$p-1$. Therefore, for any $y\in R$, we have $\deg y=|O(y)|-1\leq
p-2$. On the other hand, $|O(x)|=p-1$ and thus
we see that $R$ is of type $p-2$.
\end{expl}

\begin{expl}\label{e.6}
  Let $ n> 1$ be an odd integer. Then
                          the ring $R = \Z_n[X]/(X^2)$ is of type
                          $\varphi(n)-1$, where $\varphi(n)$ denotes
                          the Euler phi function.

As before, let us denote by $x$ the image of $X$ in $R$. Any element
in $R$ can be uniquely written as $ax+b$ with $a,b\in\Z_n$. Let
$\psi\in \aut$. Notice that $\psi(a)=a$ for all $a\in\Z_n$. Then
$\psi(x)=ax+b$ for some $a,b\in\Z_n$. Since
$\psi$ is an automorphism, there exists an element $px+q\in R$ with
$p,q\in\Z_n$ such that
$$x=\psi(px+q)=p\psi(x)+q=pa x+pb+q.$$
Thus we get $pa=1$ and so $a$ must be a unit in $\Z_n$. Further, if
$\psi(x)=ax+b$, with $a\in U(\Z_n)$,  we must also have,
$$0=\psi(x^2)=(ax+b)^2=2abx+b^2$$ and hence $2ab=0$. Since $n$ is odd
and $a$ is a unit, we have $b=0$. So, any automorphism $\psi\in\aut$
must have, $\psi(x)=ax$ for some unit $a\in R$. It is easy to see that
any such map is indeed an automorphism. Thus we see that $\aut\cong
U(\Z_n)$, which has order $\varphi(n)$. Thus, as before, we get that
$|O(y)|\leq \varphi(n)$ for all $y\in R$ and since
$|O(x)|=\varphi(n)$, we see that $R$ is of type $\varphi(n)-1$.

\end{expl}

\begin{expl}\label{e.7}
 Let Let $p$ be a prime and $n\geq 1$ be
                        any integer. Then for the direct product
                        ring $ R = \Z_{p^n}\times \Z_{p^n}\times
                        \cdots \times\Z_{p^n}$  ($k$-times), where $k<
                        p^n$,  $\aut =
                        S_k$, the symmetric group on $k$ symbols. Thus
                        $R$ is of type $k!-1$.
\end{expl}

\begin{thm} \label{t.8}
  Let $(R,M)$ be a finite local ring
                        which is not a field, such that $\deg x \leq
                        1$ for all $x\in M$. Then $\aut$
                        is an Abelian group
                        of order $2^m$, $m\geq 0$ .
\end{thm}

\begin{proof}
  Let $x\in M$. By assumption, $\deg x\leq
                         1$. Hence for any $\sigma \in \aut$,
                         $x,\sigma(x), \sigma^2(x)$ are not all
                         distinct. Thus $\sigma^2(x)=x$ for all $x\in
                         M$. Thus
                         $\sigma^2= \mathrm{id}$  on $M$. Hence by
                         \cite[Theorem 2.5]{S1},
                         $\sigma^2=\mathrm{id}$.  Therefore
                         $\aut$ is Abelian of order $2^m$, $m\geq 0$.

\end{proof}

\begin{expl}\label{r.9}
  Let $(R,M)$ be a finite local ring
                        which is not a field such that $\deg x \leq
                        n$ for all $x\in M$. Then
                        for every $\sigma\in \aut$, order of
                        $\sigma$ is $\leq (n+1)!.$
\end{expl}

\begin{thm} \label{t.10}
  Let $K$ be a perfect field of characteristic $p>
                        0$. Then $K$ is of type $n$, if and only if $K =
                        \mathbb{F}_{p^{n+1}}$.
\end{thm}

\begin{proof}
   As $K$ is of type $n$, order of  any $\sigma\in
                        \mathrm{Aut}\, (K)$ is at most
                        $(n+1)!$ and in
                        particular, the Fr\"{o}benius automorphism $\tau$ of
                        $K$ has finite order. If order of $\tau$ is $m$, then
                        $x^{p^{m}} = x$ for all $x\in K$. Hence $K$
                        is a finite field. As $K$ is of type $n$, it
                        is clear that $K = \mathbb{F}^{p^{n+1}}$. The converse
                        is obvious.
\end{proof}

\begin{cor}\label{c.15}
   Let $K$ be a field. Then $K$
                        is perfect of characteristic $p>0$ and is of
                        type $n<\infty$
                        if and only if $\Gamma_{\mathrm{Aut}\, K}(K)$
                        has finite number
                        of cliques.

\end{cor}

\begin{proof}
   The proof is immediate from Theorem \ref{t.10} and \cite[ Theorem
  1.1]{KP}.

\end{proof}

\begin{thm}\label{t.11}
Let  $ R= A_1\times A_2\times \cdots \times A_m$, where
                         $A_1, \ldots, A_m$ are local rings. Then
\begin{enumerate}
  \item If $A_i$ is not isomorphic to $A_j$ for
                         any $i\neq j$, $\aut$ is
                         isomorphic to $ \prod_{1\leq i\leq
                         m}\mathrm{Aut}\, A_i$.
\item If $m> 1$, and $A_i$ is
                         isomorphic to $A_j$ for some $i\neq j$,
                         Then $\aut\neq \mathrm{id}$. Further,  if $\aut$
                         is finite then it is of even order.
\end{enumerate}

\end{thm}

\begin{proof}
  \begin{enumerate}
    \item  Local rings have no
                         non-trivial idempotents. Hence any
                         idempotent of $R$ is of the form $ a=
                         (a_1,a_2,\ldots,a_m)$ where $a_i=0$
                         or  $a_i=1$  for each $i$. Denote by $e_i$ the
                         element
                         $$ e_i= (0,\ldots,0,1,0,\ldots,0)\in R\quad
                         i=1,2,\ldots,m
                         $$
                         where $1$ is the identity in $A_i$ and is at
                         the ith
                         place. Then $e_1,\ldots,e_m$ are $m$ pairwise
                         orthogonal idempotents in $R$
                         such that $e_1+e_2+\cdots +e_m = 1$. For
                         any $\sigma \in \aut$, $1=
                         \sigma(e_1)+\cdots+\sigma(e_m)$ and
                         $\sigma(e_1),\ldots,\sigma(e_m)$ are
                         pairwise orthogonal idempotents in $R$.
                         Thus $ \sigma(e_i)= e_j$ for some $j$, and
                         hence
$$ \sigma(A_i)=\sigma(
                         Re_i)=Re_j=A_j.$$
                         As $A_i$ is not isomorphic to $A_j$ for
                         $i\neq j$, we conclude that
                         $\sigma(e_i)=e_i$ for all $i$.
                         Therefore the restriction of $\sigma$ to
                         $A_i$ is an automorphism of $A_i$. This
                         proves the first assertion.
\item  Without  loss of generality, we may
                         assume that $A_1$ is isomorphic to $A_2$.
                         In fact, we can take $A_1=A_2$. Then the map $
                         \tau : R\longrightarrow R$ given by
$$ a=
                         (a_1,a_2,\ldots, a_m)\longmapsto
                         (a_2,a_1,\ldots, a_m)$$
                           is a non identity automorphism of $R$
                           such that $\tau ^2 = 1$. Hence the second
                         assertion follows.

  \end{enumerate}
\end{proof}

\begin{rmk}\label{r.12}
  \begin{enumerate}
   \item   If $\aut$ is of odd
                           order, then $A_i$ is not isomorphic to
                           $A_j$ for $i\neq j$.
\item  The Theorem is valid even if we
                           assume that $A_i$ has no non-trivial
                           idempotent for any $i$, instead of assuming
                           $A_i$ to be local.

  \end{enumerate}

\end{rmk}

\begin{cor}\label{c.13}
Let $R,S$ be two local rings such
                           that $R$ is not isomorphic to $S$. Assume
                           that for $a\in R$ and $b\in S$, we have
                           $\deg a=m$, and $\deg b=n$. Then for the
                           element $(a,b)\in
                           R\times S$,
$$ \deg (a,b) =  (\deg a+1)(\deg b+1)-1.$$

\end{cor}

\begin{proof}
  By the Theorem, $\mathrm{Aut}\,(R\times S)$ is
                           isomorphic to $\aut\times \mathrm{Aut}\,(S)$.
                           Therefore, it is immediate that
$$
                           \deg (a,b) = (\deg a+1)(\deg b+1)-1$$.
\end{proof}

\begin{cor}\label{c.14}
Let $R$ be a local ring of type $m$
                           and $S$ be a local ring of type $n$, $m\neq
                           n$. Then $R\times S$ is of type
                           $(m+1)(n+1)-1$.
\end{cor}

\begin{proof}
   The result is immediate from
                            Corollary \ref{c.13}.
\end{proof}
\begin{thm}\label{t.15} Let $R$ be a finite ring. Then $\Gamma_{\aut}(R)$ is planar if and only if $R$
                  is of type $\leq 3$.\end{thm}

                  \begin{proof}
                              If $\Gamma_{\aut}(R)$ is planar, then it does not contain $K_5$ by
                              [3, Theorem 11.13].
                              Hence $\deg x \leq 3$ for all $x\in R$. This proves $R$ is of type $\leq 3.$
                              The converse is clear since $K_n$ is planar for all $n \leq 4$ and $\Gamma_{\aut}(R)$
                              is a union of $K_{n's}$ .

                              \end{proof}
                              \begin{thm} For any ring $R, Aut(R)$ is a subgroup of  $Aut(\Gamma_{\aut}(R))$, but
                               the converse is not true in general. \end{thm}
                               \begin{proof} Let $ a,b\in \Gamma_{\aut}(R)$ be connected.Then there exists
                              $ \sigma \in Aut(R)$ such that $\sigma(a)= b$. Now for any $\theta\in Aut(R),
                               \theta \sigma \theta^{-1}\theta(a) =\theta(b).$ Thus the
                              direct part holds.
                              For the converse, if $R = \Z_4,$ then $Aut(R) = Id.$, but
                              $Aut(\Gamma_{\Z_4 }) = S_4$.

                               \end{proof}

                           \section{Rings with $\Gamma_{\aut}(R)$ totally
                           disconnected}

                            Let $R$ be a finite ring. In this section, we
                            shall study the structure of $R$ with
                            $\Gamma_{\aut}(R)$ totally disconnected.
                             Observe that $\Gamma_{\aut}(R)$ is
                             totally disconnected if and only if
                             $\aut =\mathrm{id}$. By \cite[Theorem
                            8.7]{AM}, any finite ring $R$ is a direct product
                             of finite local rings uniquely. As
                             $\Gamma_{\aut}(R)$ is totally disconnected,
                             each of the factor local ring has trivial
                            automorphism
                             groups. Therefore we will  study
                            structure of $R$ when $R$ is
                             local with $\aut= \mathrm{id}$.
                 \begin{thm}\label{t.15}
Let $(R,M)$ be a finite local ring such
                             that $\Gamma_{\aut}(R)$ is totally
                             disconnected. Then $R$ is isomorphic to
                             $\Z_{p^{\alpha}}$ or $\Z_2[X]/(X^2)$ where $p$
                             is a prime.
                 \end{thm}
                 \begin{proof}
                     As $\Gamma_{\aut}(R)$ is totally
                             disconnected, $\aut = \mathrm{id}$.
                             Since $R$ is a finite local ring, its
                             characteristic is
                             $p^\alpha$ for some prime $p$.
                              Then
                             $ \Z_{p^\alpha} \subset R $. Thus the
                             characteristic of $ R/M $ is $p$.

If $R=\Z_{p^\alpha}$, we have nothing to prove. So we assume that
                             $R\neq \Z_{p^{\alpha}}$.

The structure of the proof is as follows.
\begin{enumerate}
  \item We first show that there is a subring $B\subset R$ of the form
    $\Z_{p^{\alpha}}[T]/(f(T))$ where $f(T)$ is a monic polynomial
    in $\Z_{p^{\alpha}}[T]$ such that the induced map $B\to R/M$ is
    onto.
\item If $B=R$  then we show that $R$ has
  non-trivial automorphisms contradicitng our hypothesis.
\item If $B\neq R$, we choose a maximal subring $B\subset A\subsetneq
  R$ and show that $R$ has non-trivial automorphisms over $A$, again
  contradicting our hypothesis, except when $p=2$ and the only
  exception being when $R=\Z_2[X]/(X^2)$.

\end{enumerate}

We  show that there is a subring $B \subset R$ of the form
$\Z_{p^\alpha}[a]$ such that the natural map $B\to R/M$ is onto. If
$R/M= \Z_p$, we may take $B=\Z_{p^\alpha}$. So, let us assume that
$R/M\neq \Z_p$.
As  $\Z_{p}$ is a
                             perfect field and $R/M$ is a finite
                             separable extension of $ \Z_p $, $ R/M $ is
                             a simple field extension of $ \Z_p $ and thus
                               $ R/M =
                             \Z_p[\obar{x}]$ for some element
                             $0\neq \obar{x} \in R/M$. Let $ f_1(T)$
                             be the irreducible polynomial of
                             $ \obar{x} $ over $ \Z_p $.  Choose
                             $f(T)\in R[T]$, a monic polynomial, such
                             that $ f_1(T)$ is the image of $ f(T) $ in
                             $ R/M[T]$. Since $\obar{x}$ is separable
                 over $\Z_p$, by Hensel's Lemma, there
                             exists a lift $ a\in R $ of $\obar{x}$
                 such that $f(a)=0$.
Denote by  $ B$ the subring $\Z_{p^\alpha}[a]$ of $R$. It is clear that the
natural map $B\to R/M$ is onto.

Next  We claim that  $ \Z_{p^\alpha}[T]/(f(T)) $ is isomorphic
                             to  $B$.
                              Consider
                             the natural $\Z_{p^\alpha}$- epimorphism:
$$ \theta :\Z_{p^\alpha}[T]\longrightarrow B, \quad T\mapsto a. $$

  Then, clearly $f(T)\in
                             \mathrm{Ker}\, \theta$. Hence $\theta$ induces an
                             epimorphism:
$$ \obar{\theta}:
                             \Z_{p^\alpha}[T]/(f(T))\longrightarrow
                             B.$$
 Notice
                             that, as $\Z_p[T]/(f_1(T))$ is a field,
                             $\obar{p}$, the image of $p$ in $\Z_{p^\alpha}$,
                             generates the unique maximal ideal in
                            $ \Z_{p^\alpha}[T]/ (f(T))$.  Consequently,  every
                            ideal in  $ \Z_{p^\alpha}[T]/ (f(T))$
                            is generated by a power of $\obar{p}$.
In particular so is $\mathrm{Ker}\,\obar{\theta}$ and so let this ideal
                             be $(\obar{p}^k)$ for some integer $k$. Then
                            $\obar{\theta}(\obar{p}^k) = \obar{p}^k= 0 $ in
                            $R$. This implies $k=\alpha$. Thus
                            $\mathrm{Ker}\, \obar{\theta}=0$. Hence
                            $\obar{\theta}$ is an isomorphism proving
                             our claim.

We, now, consider
                           the case $B = R$. In this case $R$ is
                           isomorphic to $\Z_{p^\alpha}[T]/(f(T))$ and
                           since we have assumed that $R\neq
                           \Z_{p^\alpha}$, we see that the monic
                           polynomial $f$ has degree greater than one.
                            Its image $f_1(T)$
                           in $\Z_p[T]$ is an irreducible
                           polynomial. Consider the Fr\"{o}benius
                           automorphism $\tau$ of
                           $\Z_p[T]/(f_1(T))=R/M$.  Since  $\deg f_1(T)>1$ the automorphism
                             $\tau$ can not be identity.

For any automorphism $\beta$ of
                           $\Z_p[T]/(f_1(T))$, the composite map
                            $$ \Z_p[T] \stackrel{\pi}{\longrightarrow}
                            \Z_p[T]/(f_1(T))\stackrel{\beta}\backsimeq
                            \Z_p[T]/(f_1(T))$$
                            is onto and if $\beta\circ\pi (T) =
                            u$, then $f_1(u)= 0$. We know that
                            $f_1(T)$ is irreducible over $ \Z_p $.
                            Hence $u$ is a simple root of $f_1(T)$. We have
                            $ f(X)\in \Z_{p^{\alpha}} [X] \subset R[X]$,
                            and its image is $f_1(X)$ in
                            $\Z_p[X]\subset R/M[X]$. As seen above,
                            by Hensel's Lemma, there exists a lift  $a\in R$
                            of $u$ such that $f(a)=0$.  Then consider
                            the homomorphism:
$$
                            \psi: \Z_{p^\alpha}[T]\rightarrow R=
                            \Z_{p^\alpha}[T]/(f(T))\quad T\mapsto a $$
                            Since $f(a)=0$, this  map  induces
                            an  endomorphism $$ \obar{\psi}:
                            R=\Z_{p^\alpha}[T]/(f(T))\longrightarrow
                            R=\Z_{p^\alpha}[T]/(f(T))$$
                            and the diagram :
                               $$\begin{array}{ccc}
                               \Z_{p^\alpha}[T]/(f(T))&
                               \stackrel{\obar{\psi}}\longrightarrow &
                               \Z_{p^\alpha}[T]/(f(T))\\
                             \downarrow &   & \downarrow \\
                             \Z_p[T]/(f_1(T))&
                             \stackrel{\beta}\backsimeq

                              &
                             \Z_p[T]/(f_1(T))
                             \end{array}$$
is commutative. As
                             $\beta$ is obtained from $\obar{\psi}$
                             after tensoring with $\Z_p$ over
                             $\Z_{p^\alpha}$, $\obar{\psi}$ is onto.
                             Hence, as $R$ is finite, $\obar{\psi}$
                             is an automorphism. Finally, taking
                             $\beta=\tau$ and since $\tau\neq
                             \mathrm{id}$,  $\obar{\psi}\neq
                            \mathrm{ id}$. Thus we arrive at a
                             contradiction to our hypothesis that
                             $\aut$ is trivial, in this case.

Lastly, we look at the case when $B\neq R$. Then we may choose a subring
                            $A$ of $R$,  with  $B\subset
                            A$, $A\neq R$ and maximal with respect to this
                            property. Then $A$ is a local ring with
                            maximal ideal $M_A= M\bigcap A$, and
                            $R=A[\lambda]$ for every $\lambda \in
                            R-A$.

Since $B$ maps onto $R/M$, so does $A$. If $M\subset A$, and in
                            particular, if $M=M_A$, then this would
                            force $A=R$, which is not the case. So,
                            $M\neq M_A$.

Since $R$ is a finitely
                            generated module over $A$, by Nakayama's
                            lemma, we also have $M_AR+A\neq R$. But,
                            $A\subset M_AR+A\subsetneq R$ and $M_AR+A$
                            is naturally a subring of $R$ and thus by
                            maximality, we must have $A=M_AR+A$ and
                            thus $M_AR\subset
                            A$. Since $1\not\in M_AR$, and $M_A\subset
                            M_AR\subsetneq A$, we see that
                            $M_AR= M_A$. So, we have shown,

\begin{equation}\label{eq1}
               M_AR=M_A\subsetneq M
   \end{equation}

                            Choose $ \lambda\in M-M_A$ such that
                            $\lambda^2\in A$. This can always be
                            done as elements of $M$ are nilpotent.
                            Thus $R=A[\lambda]$ where $\lambda \in
                            R-A$ and $\lambda^2\in A$ and in fact in
                            $M_A$. Now, consider
                            the $A$-algebra epimorphism:
$$  \psi: A[T]\longrightarrow R,\quad  T\mapsto \lambda.$$

 One clearly has $\psi(T^2-\lambda^2)=0$. Similarly, for any element
 $a\in M_A$, $a\lambda\in M_A$ by equation (\ref{eq1}) above. Thus
 we see that,
$$\mathrm{Ker}\, \psi\supset (T^2-\lambda^2, aT-a\lambda)=J$$
where $a$ runs through elements of $M_A$.

We claim that the above inclusion is an equality. If
                           $f(T)\in
                          \mathrm{Ker}\,\psi$, then, we can write
$$  f(T) =
                          (T^2-\lambda^2) g(T) + aT-b $$

                          where $g(T), aT-b\in A[T]$. By assumption,

                            $$  0  = f(\lambda)=a\lambda-b.$$

                         This forces  $a$ to be in $M_A$, since
                         otherwise $a$ is a
                         unit, and in that case
                         $\lambda =a^{-1}(a\lambda)= a^{-1}b\in A$
                         contradicting our choice of $\lambda$.
                         Thus $aT-b=aT-a\lambda\in J$ establishing our claim.
Thus we have,

$$  \obar{\psi}: A[T]/J\simeq R .$$

 Let $\mathbf{a} $  be the socle of $A$. If $\mathbf{a} =A$, then $A$
 is a field. From the above isomorphism, we have $R=A[T]/(T^2)$ since
 $\lambda^2\in M_A=0$ and $aT-b=0$ since $a,b\in M_A=0$ and thus
 $J=(T^2)$. If $u\in A$ is a unit, then $T\mapsto uT$ gives an
 automorphism of $R$ and it is non-trivial if $u\neq 1$. So, we may
 assume that $1$ is the only unit in $A$ and then $A=\Z_2$, leading us
 to the exception mentioned in the theorem.

So, from now on, let us assume that $\mathbf{a}\subset M_A$. Now, we
show that $R$ has a non-trivial automorphism as $A$-algebras, proving
the theorem.

Define an ideal $I$ of $A$ by,
                         $$I = (0:\lambda)_A = \{ x\in A\mid x\lambda =
                         0\}.$$
                    Since     $ \lambda \neq 0$ clearly  $ I\neq A$
                         and hence $I\subset M_A$.
We look at two cases, either  $\mathbf{a}$ is contained in  $I$ or
                         not. First we consider the case when
                         $\mathbf{a}\subset I$. Let $0\neq v\in
                         \mathbf{a}$ and consider the $A$-algebra
                         automorphism,
$$\alpha: A[T]\to A[T],\quad T\mapsto T+v.$$

We want to show that $\alpha$ respects the ideal $J$. We have,
\begin{align*}
  \alpha(T^2-\lambda^2)&=(T+v)^2-\lambda^2\\
&=(T^2-\lambda^2)+2vT+v^2\\
&=(T^2-\lambda^2)+(2vT-2v\lambda)+2v\lambda+v^2\\
&=(T^2-\lambda^2)+(2vT-2v\lambda)
\end{align*}
since $v^2=0$ because $v\in\mathbf{a}\subset M_A$ and $2v\lambda=0$
since $v\in I$. Thus $\alpha(T^2-\lambda^2)\in J$. Similarly, for
$a\in M_A$,
\begin{equation*}
  \alpha(aT-a\lambda)=a(T+v)-a\lambda=aT-a\lambda+av=aT-a\lambda
\end{equation*}
since $av=0$. Thus, $\alpha(aT-a\lambda)\in J$. So, we get an induced
surjective $A$-algebra homomorphism,
$$\obar{\alpha}: R=A[T]/J\to A[T]/J=R,$$
which then must be an automorphism. Since $T\mapsto T+v$ and $v\neq
0$, this is a non-trivial automorphism.

Lastly, we consider the case when the socle is not contained in
$I$, but the socle is contained in $M_A$. Then choose an element $v$
in the socle not contained in $I$. Consider the $A$-algebra
automorphism
$$\beta: A[T]\to A[T],\quad T\mapsto (1+v)T.$$
As before, we proceed to check that this map respects the ideal $J$.
\begin{align*}
  \beta(T^2-\lambda^2)&=(1+v)^2T^2-\lambda^2\\
&=(T^2-\lambda^2)+2vT^2+v^2T^2\\
&=(T^2-\lambda^2)+ 2v(T^2-\lambda^2)+2v\lambda^2+v^2T^2\\
&=(1+2v)(T^2-\lambda^2)
\end{align*}
since $v^2=0$ and $v\lambda^2=0$ by virtue of the fact that $v$ is in
the socle as well as in $M_A$ and $\lambda^2\in M_A$. So,
$\beta(T^2-\lambda^2)\in J$.

Similarly, for any $a\in M_A$ one has,
$$\beta(aT-a\lambda)=a(1+v)T-a\lambda=(aT-a\lambda)+avT=aT-a\lambda,$$
since $av=0$. Thus $\beta(aT-a\lambda)\in J$. So, we get an induced
$A$-algebra surjection,
$$\obar{\beta}: R\to R,$$
which  is an isomorphism. Further, since
$\obar{\beta}(\lambda)=\lambda+v\lambda$ and $v\lambda\neq 0$ since
$v\not\in I$, this is a non-trivial automorphism.

This concludes the proof of the theorem.

 \end{proof}

                             \begin{cor}\label{t.16}
Let $R$ be a finite ring
                             such that
                             $\Gamma_{\aut}(R)$ is totally disconnected.
                             Then $R$ is  a finite product of rings
                             of the type $\Z_{p^\alpha}$  and
                             $\Z_2[X]/(X^2)$.

                 \end{cor}

                 \begin{proof}
                    Since $R$ is a finite ring, by
                            \cite[Theorem 8.7]{AM},  $R$ is a finite
                             product of local rings. Further, as
                            $\Gamma_{\aut}(R)$
                             is totally disconnected, $\aut =
                             \mathrm{id}$. Hence each of the local ring in
                             the decomposition of $R$ has
                             automorphism group trivial. Therefore
                             the result follows from Theorem \ref{t.15}.
                 \end{proof}
                 \begin{rmk}
                   Let $(R, M)$ be finite local ring with
                   characteristic of $R/M=p$. If
                   $[R/M:\F_p]>2$, then $R$ is of type at
                   least $2$. This can be deduced from the
                   proof of  Theorem \ref{t.15}.
                 \end{rmk}

                             \section{ Some connected subsets of
                             $\Gamma_{\aut}(R)$}

                               In this section, we study the
                               structure of a finite local ring $R$
                               for which certain subsets of  $\Gamma_{\aut}(R)$
                               are connected.
                   \begin{thm}\label{t.17}
Let $(R,M)$ be a finite local
                               ring and $U(R)$ be the set of units of
                               $R$. If $U(R)-\{1\}$ is a connected
                               subset of $\Gamma_{\aut}(R)$, then $R$ is
                               one of the following.
                   \begin{enumerate}
                 \item $\Z_2,\Z_3,\Z_4$ or $\F_4$.
\item  $\Z_2[X_1,\cdots,X_m]/I$ where $I$ is
                               the ideal of $ \Z_2[X_1,\cdots,X_m]$
                               generated by $\{X_iX_j| 1\leq i, j
                               \leq m\}$.
                   \end{enumerate}

                   \end{thm}

                   \begin{proof}
                 If $U(R) = \{1\}$, then $M=0$ since
                               $1+x$ is a unit for all $x\in M$.
                               Therefore, in this case,  $R=\Z_2$.

Now assume $U(R)-\{1\} \neq \emptyset$. Let $ p^n $
                               be the characteristic of $ R $ so that
                               $\Z_{p^n}\subset R$. The number of
                               units in $\Z_{p^n}$ is $p^{n-1}(p-1)$.
                               For any $\sigma\in \aut$,
                               $\sigma$ is identity on
                               $\Z_{p^n}$. Thus all elements of
                               $U(\Z_{p^n})\subset U(R)$ have orbits
                               consisting of just one element. If
                               $U(R)-\{1\}$ is connected, it follows
                               that the cardinality of $U(\Z_{p^n})$
                               can not be greater than two. Thus
                               $p^{n-1}(p-1)\leq 2$.
                               We deduce that
                               either $p=2,n=1,2 $ or $ p=3,
                               n=1$.

If $M\supsetneqq
                               \obar{p}\Z_{p^n}$, then for any $ x\in
                               M$, with  $x\not\in \obar{p}\Z_{p^n}$, $1+x $ is
                               a unit not in  $\Z_{p^n}$.
                               Therefore, in the cases $ p=2,n=2 $ or
                               $ p=3,n=1$, one sees that $U(R)-\{1 \}$
                               is not connected.
                               Consequently $ M = \obar{p}\Z_{p^n} $,
                               if $ p=2,n=2 $ or $ p=3, n=1 $.

Let us first look at the cases, $p=2,n=2$ and $p=3,n=1$.
In these cases, $ M =
                                     \obar{p}\Z_{p^n} $ from above.
                     The set $U(\Z_{p^n})-\{1\}$ has exactly
                     one element and it is invariant
                     under all automorphisms of
                     $R$. Thus, this single element
                     set is a connected component of
                     $U(R)-\{1\}$, and since this set
                     is assumed to be connected, we
                     see that
                     $U(R)-\{1\}=U(\Z_{p^n})-\{1\}$.
                     This implies
                    $ \Z_{p^n}-\obar{p}\Z_{p^n} = R-M
                     $,
                                      Thus $R = \Z_{p^n}$ proving the
                      theorem in these cases.

We are left with the last case, when  $ p=2$ and $ n=1$.
                                    In this case $\Z_2\subset R$.
If $R$ is a field, then $R=\F_q$ where $q=2^s$. The automorphism group
                                    of $\F_q$ has order $s$ and thus
                                    the orbits have cardinality at
                                    most $s$. Since the cardinality of
                                    $U(\F_q)$ is $q-1$, we get that
                                    $2^s-2=q-2\leq s$. One easily sees that
                                    this implies $s\leq 2$. Since we
                                    are assuming that $U(R)-\{1\}\neq
                                    \emptyset$, this forces $s=2$ and
                                    $R=\F_4$. One easily checks that
                                    in this case, $U(R)-\{1\}$ is
                                    connected.

Finally we may assume that $M\neq 0$.   Let $0\neq x\in M$. If $u\neq 1$ is
any unit, then the connectedness of $U(R)-\{1\}$ implies that there
exists a $\sigma\in\aut$ such that $\sigma(1+x)=u$ and hence $u\equiv
1\bmod M$. This implies $R/M\cong\Z_2$. Next we show
                                    that $M^2=0$. For
                                    this it suffices
                                    to show that for any $x\in M-M^2$
                                    and any $y\in M$, $xy=0$. If
                                    $xy\neq 0$, then there exists an
                                    automorphism $\tau$ of $R$ so that
                                    $\tau(1+x)=1+xy$ which
                                    implies that
                                    $\tau(x)=xy$. But, then
                                    $\tau(x)\in M-M^2$ and $xy\in
                                    M^2$, which is a
                                    contradiction. So $M^2=0$.

 Now, let
                                    $\{a_1,\cdots,a_m\}$ be a
                                    minimal set of generators for
                                    $M$. Then consider the
                                    surjective homomorphism
 $$ f:
                                    \Z_2[X_1,\cdots,X_m]\rightarrow
                                    R,\quad  X_i\longmapsto a_i$$
                                    As $M^2 =0$, $\abs{M} = 2^m$,
                                    since $m= \dim_{R/M} M/M^2$ and
                                    $R/M= \Z_2$. Therefore $\abs{R}=
                                    2\abs{M} = 2^{m+1}$  and
                                    $\mathrm{Ker}\, f =
                                    I$ is the ideal generated by
                                    $X_iX_j$ with $1\leq i,j\leq m$.  As
                                    $\abs{\Z_2[X_1,\cdots,X_m]/I}=
                                    2^{m+1}$, it follows that
                                    $\Z_2[X_1,\cdots,X_m]/I$ is
                                    isomorphic to $R$. It is easy to
                                    see that in this case,
                                    $U(R)-\{1\}$ is indeed connected.
                   \end{proof}

                   \begin{thm}\label{t.18}
Let
                                    $(R,M)$ be a finite local
                                    ring with characteristic $p^n$.
                                    If $M-\{0\}$ is connected, then
                                    $R=\Z_4$  or
                                    $\F_q[X_1,\cdots, X_m]/ I$ where
                                    $\F_q$ is a finite field with
                                    $q$ elements and $ I$ is the ideal
                                    generated by elements of the form
                                    $X_iX_j$ with $ 1\leq i,j \leq
                                    m$. By convention, we will include
                                    the case $R=\F_q$, when $m=0$.
                   \end{thm}

                               \begin{proof}
If $M-\{0\}=\emptyset$, then $R$ is a field and hence $\F_q$ for some
$q$. So, let us assume that
$M\neq 0$.

                 As characteristic of $R$ is
                                    $p^n$, $\Z_{p^n} \subset R$.
                                    Exactly as  in Theorem \ref{t.17},
                                    we can see that
                                    $M^2 = 0$. Now, note that
                                    $M\cap \Z_{p^n} = (\obar{p})$.
                                    Hence $n\leq 2$.

                                   First we  consider the case  $n=2$.
                                         In this case, if $p> 2$,
                                         then for any $1< u < p$,
                                         the two elements $u\obar{p},
                                         \obar{p}$ are
                                         distinct non-zero  elements of $M$
                                         and for any $\sigma \in
                                         \aut$, $\sigma(\obar{p}) =
                                         \obar{p}$ and $ \sigma( u\obar{p})=
                                         u\obar{p}$. This contradicts
                                         the fact $M-\{0\}$ is
                                         connected. Hence $p=2$. In
                                         this case $M= \{\obar{2},
                                         0\}$ since $\sigma(\obar{2}) =
                                         \obar{2}$ for any
                                         automorphism $\sigma$ of $R$
                                         and $M-\{0\}$ is
                                         connected.  If $R \neq \Z_4$,
                                         then choose $
                                         \lambda\in R-\Z_4$. Clearly
                                         $\lambda\not\in M$ and hence
                                         is a unit. Now, note that
                                         $\obar{2}$ and $ \lambda
                                         \obar{2}$ are in in $M=
                                         \{\obar{2},0\}$.
                                         Therefore
$\lambda\obar{2}=\obar{2}$ and hence $(\lambda-1)\obar{2}=0$. Since
                                         $\obar{2}\neq 0$, this
                                         implies that $\lambda-1\in M$
                                         and and since $M\subset
                                         \Z_4$, we see that
                                         $\lambda\in \Z_4$,
                                         contradicting our choice of
                                         $\lambda$. Thus, in this case
                                         $R=\Z_4$.

In the last case of $n=1$, we have $\Z_p\subset R$. So
                                                $\Z_{p}\subset R/M$ is
                                                a finite separable
                                                extension and so as in
                                                Theorem \ref{t.15},
                                                there exists a
                                                finite field $\F_q
                                                \subset R$ such that
                                                $\F_q$ is isomorphic
                                                to $R/M$. Now, let
                                                $\{a_1,\cdots,a_m \}$
                                                be a minimal set of
                                                generators for $M$.
                                                Then consider as
                                                before the
                                               surjective  map
$$f:
                                                \F_q[X_1,\cdots,X_m]\rightarrow
                                                R,\quad   X_i\longmapsto
                                                a_i$$

                                               Again  $\mathrm{Ker}\, f$ is the
                                                ideal $I$ generated by
                                                elements of the form $
                                                X_iX_j$ with $ 1\leq i,
                                                j\leq m$.  Note
                                                that, as seen above,
                                                $m = \dim_{R/M}M$.
                                                Thus $\abs{R} =
                                                q^{m+1}$ and similarly
                                                $\abs{\F_q[X_1,\cdots,X_m]/I}
                                                = q^{m+1}$.
                                                Consequently $f$ is
                                                an isomorphism.
                                                Hence the proof is
                                                complete.
                   \end{proof}

                               \begin{thm}\label{t.19}
Let $K/E$ be a
                                                field extension, and
                                                let $\mathrm{Aut}_E K = H$.
                                                Assume $K-E\subset
                                                \Gamma_H(K)$ is
                                                connected. Then
                                                either $K/E$ is
                                                algebraic or all
                                                elements of $K-E$
                                                are transcendental
                                                over $E$. Moreover,
                                                $K^H = E$. Further,
                                                if $K/E$ is
                                                algebraic and not
                                                equal, then $E=\F_2$
                                                and $K=\F_4$.

                   \end{thm}

                               \begin{proof}
                 Let $a,b\in
                                                K-E $ be two
                                                distinct elements
                                                such that $a$ is
                                                algebraic over $E$.
                                                Since $K-E\subset
                                                \Gamma_H(K)$ is
                                                connected, there
                                                exists $\sigma \in
                                                H$ such that
                                                $\sigma(a) = b$.
                                                Therefore $b$ is
                                                also algebraic over
                                                $E$. This proves the
                                                first part of the
                                                statement.

 Next,
                                                note that $E\subset
                                                K^H$. Then, as $K-E\subset
                                                \Gamma_H(K)$ is
                                                connected, it is
                                                clear that $K^H - E
                                                = \emptyset $, or in
                                                other words  $K^H
                                                = E$.

Now, let $K/E$
                         be algebraic. We  shall consider the  cases
                         of $K$ being infinite or finite separately.

First consider the case when $K$ is infinite. If $K-E\neq \emptyset$,
                                                 let $\lambda\in
                                                 K-E$. Let
                                                 $p(T)$ be the
                                                 irreducible
                                                 polynomial of
                                                 $\lambda$ over
                                                 $E$. Then for any
                                                 $\sigma\in H$,
                                                 $\sigma(\lambda)$
                                                 must be a root of
                                                 $p(T)$ and in
                                                 particular the orbit
                                                 of $\lambda$ is
                                                 finite. Since $K-E$
                                                 is connected, this
                                                 means that $K-E$ is
                                                 the orbit of
                                                 $\lambda$ and thus
                                                 $K-E$ is a finite
                                                 set. Thus, $K$ is a
                                                 finite dimensional
                                                 vector space over $E$
                                                 and so $E$ must be
                                                 infinite too. For any
                                                 $0\neq a\in E$,
                                                 $a\lambda\in K-E$ and
                                                 these are
                                                 distinct. So, $K-E$
                                                 is infinite, which is
                                                 a contradiction. So,
                                                 $K$ can not be infinite.

Next, let us consider the case when   $K$ is finite. Let $E=\F_q$ and
let $|K:\F_q|=t>1$. Then  $H$ is a cyclic group of order $t$ generated
by an appropriate power of the Frobenius. For any $\lambda\in K-E$,
the cardinality of the orbit of $\lambda$  is therefore at most
$t$. Since $K-E$ is connected, we have $|K-E|\leq t$. On the other
hand, $|K-E|=q^t-q$ and thus we get $q^t-q\leq t$. It is easy to check
that this can happen only when $q=2$ and $t=2$. This proves the
theorem.

If $E=\F_2$ and $K=\F_4$, then it is trivial to check that $\F_4-\F_2$
is indeed connected.

                   \end{proof}

                  Let $K/k$ be a field extension where
                 $K$ and $k$ are algebraically closed. Let $ H =
                 \mathrm{Aut}_k(K)$. Then, it is easy to check that
                 $K-k\subset \Gamma_H(K)$ is connected. We, now,
                 ask  the converse:
         \begin{que}
            Let $k$ be an algebraically closed
                 field and let $K/k$ be a field extension with $H =
                 \mathrm{Aut}_k(K)$. Assume  $K-k\subset \Gamma_H(K)$ is
                 connected. Is $K$ algebraically closed?
         \end{que}

                 This question is a slight variant of part of
                 Conjecture 2.1 in \cite{KP}.
                 
                After the work was done, Roger Wiegand pointed out that Theorem 3.1 and 
                Corollary 3.2 have also been proved in \cite{KS}. However, the proof is 
                different.


\end{document}